\documentclass[nospthms,numbers,natbib,graybox
]{svmult}
\usepackage[utf8]{inputenc}
\usepackage{amsfonts}
\usepackage{amssymb,amsmath,amsthm,stmaryrd,enumerate,mathtools}
\let\I\undefined

\usepackage{bbold}
\def\Isymb{{\mathbb 1}}

\makeatletter

\newcommand{\I}[1][]{\Isymb
        \@ifempty{#1}{\isparamm\subscr\subnozjel}{\subnozjel{#1}}} 
\def\subscr#1{_{(#1)}}
\def\subnozjel#1{_{#1}}

\@ifundefined{texexp}{\let\texexp\exp\def\exp{\texexp\isparam\cbrstar}}\relax

\newcommand{\cond}[1][]{\,#1|\,}
 
\def\cbrstar{\cbr*}

\def\isparam#1{\isparamm#1\relax}

\def\isparamm{\@ifnextchar{\bgroup}}

\newcounter{enumn}\newcounter{enumr}

\DeclarePairedDelimiter\abs{\lvert}{\rvert}
\DeclarePairedDelimiter\cbr{\{}{\}}
\DeclarePairedDelimiter\norm{\|}{\|}
\DeclarePairedDelimiter\br{\lbrack}{\rbrack}
\DeclarePairedDelimiter\zjel{(}{)}
\DeclarePairedDelimiter\q{\langle}{\rangle}

\def\PEzjel{\zjel*}
\def\PEbr#1{\br*}

\def\PEfont{\mathbf}
\newcommand{\PE}[1][]{\PEfont{\Pe}_{#1}%
  \@ifnextchar^{\xPE}{\isparam\PEzjel}}
\def\xPE^#1{^{#1}\isparam\PEzjel}
\renewcommand{\E}{\def\Pe{E}\PE}
\renewcommand{\P}{\def\Pe{P}\PE}

\def\F{{\mathcal F}}

\let\smallset=\cbr
\def\set@internal#1#2#3{#1{#2\,:\,#3}}
\def\set{\@ifstar{\set@internal{\cbr*}}{\set@internal\cbr}}

\let\sign=\sgn

\def\real{{\mathbb R}}

\def\@bs{\relax}
\def\stripbs#1#2\relax#3{%
  \def\next{#1}%
  \ifx\@bs\next\relax\expandafter\@firstoftwo\else\expandafter\@secondoftwo\fi
  {\def#3{#2}}{\def#3{#1#2}}%
}

\def\setbs#1#2\relax{\def\@bs{#1}}
\expandafter\stripbs\string\ \relax\@bs
\expandafter\setbs\@bs\relax\relax

\def\defname#1#2#3#4{%
  \expandafter\stripbs\string#1\relax\nn
  \edef\nn{\csname#3\nn#4\endcsname}%
  \expandafter\def\nn#2\relax%
}

\def\deftx#1{%
  \defname{#1}{{\tilde{#1}}}{t}{}%
}

\newcommand\defxn[2]{%
  \expandafter\newcommand\csname#1n\endcsname[1][n]{#2^{(##1)}}}

\makeatother

\usepackage[mathcal]{euscript}
\usepackage{graphicx}
\usepackage{xcolor}
\definecolor{labelkey}{named}{teal}
\usepackage[colorlinks,citecolor=blue]{hyperref}

\def\levy/{L\'evy}
\newtheorem{theorem}{Theorem}
\newtheorem{lemma}[theorem]{Lemma}

\newtheorem{proposition}[theorem]{Proposition}

\theoremstyle{definition}

\theoremstyle{remark}

\newtheorem*{remark*}{Remark}

\let\PEfont\mathbf

\def\W{\mathbf{W}}
\def\cH{\mathcal{H}}
\def\cT{\mathcal{T}}
\def\lT{\mathbf{T}}

\defxn hh
\defxn b\beta 
\defxn F\F
\defxn MM
\defxn ZZ
\defxn {xi}\xi
\deftx\beta \defxn {tb}\tbeta

\def\e{\mathcal{E}}

\begin{document}
\title*{On 
  the exactness of 
  the \levy/--transformation}
\author{Vilmos Prokaj}
\institute{
  E\"otv\"os Lor\'and University, Department of Probability Theory and
  Statistics, 1117 Budapest, P\'azm\'any P. s\'et\'any 1/C, Hungary, 
  \email{prokaj@cs.elte.hu}} 

\maketitle

\abstract{In a recent paper we gave a sufficient condition for the
  strong mixing property of the \levy/--transformation. In this note we
  show that it actually implies a much stronger property, namely 
  exactness.
}

\section{Introduction}
Our aim in this short note, to supplement the result 
of \cite{Prokaj2010a}. In that work we obtained a condition which implies
the strong mixing property, hence the ergodicity of the 
\levy/--transformation. We reformulate this condition, see
\eqref{cond:newform}  below,
 and show that
it actually implies a stronger property called
exactness. That is, we deduce that the tail $\sigma$-algebra of the
\levy/ transformation is trivial provided that condition \eqref{cond:newform}
holds.

\section{Summary of the results of \cite{Prokaj2010a}}
\label{sec:3}

First, we  fix some notations. 
$\W=C[0,\infty)$ is the space of continuous function defined on
$[0,\infty)$, $\P$ is 
the Wiener measure on the Borel $\sigma$-field of $\W$, and $\beta$ is the
canonical process on $\W$. 
Finally $T$ is a $\P$
almost everywhere defined transformation of $\W$ defined by the formula 
\begin{equation}\label{eq:T def}
  (T\beta)=\int h(s,\beta)d\beta_s
\end{equation}
where $h$ is a progressively measurable function on $[0,\infty)\times \W$
taking values in $\smallset{-1,1}$. We use the notation $\bn$ for $T^n
\beta$ and 
$(\Fn_t)_{t\geq 0}$ for the filtration generated by $\bn$ and
$\hn_s=\prod_{k=0}^{n-1} h(s,\bn[k])$. 

The transformation $T$ is called {\em exact}, whenever $\bigcap_n\Fn_\infty$
is trivial. 

\bigskip

The \levy/ transformation is
obtained by the choice $h(s,\beta)=\sign (\beta_s)$ and denoted by
$\lT$. The rest of this section is devoted to this special case.

\bigskip

The main observation of \cite{Prokaj2010a} was that the existence of
certain stopping times makes it possible to estimate the covariance of
$\hn_s$ and $\hn_1$, which is the key to prove the strong mixing
property of $\lT$. More precisely, 
for $r\in(0,1)$ and $C>0$ let 
\begin{displaymath}
 \tau_{r,C}=\inf\set*{s>r} {\exists n,\,
  \bn_s=0, \min_{0\leq k<n}\abs{\bn[k]_s}>C\sqrt{(1-s)_+}}.
\end{displaymath}
That is $\tau_{r,C}$ is the first time after $r$ when for some $n$ the
first $n$ iterated paths are relatively far away from the origin while
$\bn$ is zero. 

Then it was proved that
\begin{equation}\label{eq:cov hn}
  \limsup_{n\to\infty}\abs*{\E{\hn_r\hn_1}}\leq \P{\tau_{r,C}=1}+\P{
  \sup_{0\leq s\leq 1}\abs{\beta_s}>C}.
\end{equation}
It was stated without the first term on the right, under the
assumption that this term is zero. The proof of this inequality used the
coupling of the shadow path $\tbeta$, reflected after $\tau_{r,C} $ and the
original path $\beta$.  This argument actually yields the following
form of \eqref{eq:cov hn}
\begin{equation}
  \label{cond:newform}
  \lim_{n\to\infty}\abs*{\E{\hn_1\cond[\big]\Fn_1\vee\Fn[0]_r}}\leq 
  \P{\tau_{r,C}=1}+\P{\sup_{0\leq s\leq 1}\abs{\beta_s}>C}.
\end{equation}
Note that the limit on the left hand side exists as
$\abs*{\E{\hn_1\cond[\big]\Fn_1\vee\Fn[0]_r}}$
is a reversed submartingale.

By virtue of the estimates in \eqref{eq:cov hn} and
\eqref{cond:newform} a sufficient condition for the strong mixing of
the \levy/ transformation is that
\begin{equation}\label{eq:tau}
  \tau_{r,C}<1,\quad \text{almost surely, for all $r\in(0,1)$,
    $C>0$.} 
\end{equation}

The main result of this paper is the following theorem.
\begin{theorem}\label{thm:main}
  If \eqref{eq:tau} holds then the \levy/ transformation is exact.
\end{theorem}
The proof is based on the estimate \eqref{cond:newform} and is given
in the next section where we do not assume the special form of the
\levy/ transformation. That is, we prove the next statement from which
Theorem \ref{thm:main} follows.

\goodbreak

\begin{proposition}\label{prop:2}
  Let $T$ be the transformation of the Wiener--space as in \eqref{eq:T
    def}. 
  If 
  \begin{equation}
    \lim_{n\to\infty}\E{\hn_t\cond[\big]\Fn[0]_{rt}\vee\Fn_t}=0,\quad\text{for almost
      all $t>0$ and $r\in[0,1)$}
    \label{eq:cond}
  \end{equation}
  then $T$ is exact.
\end{proposition}

\section{Proof of Proposition \ref{prop:2}.}
\label{sec:2}

For a deterministic function $f\in L^2([0,\infty))$
we will use the notation $\e(f)$ for
\begin{displaymath}
  \e(f)=\exp{\int_0^\infty f(s)d\beta_s-\frac12\int_0^\infty f^2(s)ds}.
\end{displaymath}
Since the linear hull of the set of $\set*{\e(f)}{ f\in
  L^2([0,\infty))}$
is dense in $L^2(\W)$ the following
statement is obvious.
\begin{proposition}\label{prop:1}
  $\bigcap_n \Fn_{\infty}$ is trivial if and only if 
  $\E{\e(f)\cond[\big]\Fn_\infty}\to1$ 
  for each $f\in L^2([0,\infty))$. 
\end{proposition}

To express $\E{\e (f)\cond[\big]\Fn_\infty}$ we use the next proposition.
\begin{lemma}\label{lem:2} Assume that $\xi$ is a measurable and
  $\Fn[0]$--adapted 
  process satisfying 
  $\E{\int_0^\infty \xi^2_sds}<\infty$. Then 
  \begin{displaymath}
    \E{\int_0^\infty \xi_s d\bn[0]_s\cond[\bigg]\Fn_\infty}=\int_0^\infty
    \E{\xi_s\hn_s\cond[\big]\Fn_s}d\bn_s 
  \end{displaymath}
\end{lemma}
\begin{proof}
First observe that both sides  of the equation makes sense. 

Denote by $V$ the left hand side of the equation and by $V'$ the right
hand side.
Besides let $U\in L^2(\Fn_\infty)$ and write it, using
that $\Fn$ is generated by the Brownian motion $\bn$, as
$U=c+\int_0^\infty u_sd\bn_s$ with some $c\in\real$ and 
$\Fn$--predictable $u$. Then
\begin{align*}
  \E{UV}&=\E{\int_0^\infty \xi_s\hn_s u_s ds}\\ &=
  \E{\int_0^\infty\E{\xi_s\hn_s|\Fn_s}u_sds}=
  \E{UV'}.
\end{align*}
  This proves that $V=V'$ which is the claim.
\end{proof}

In the proof 
of the next statement 
we call a probability measure
$Q\sim \P$ simple when it is in the form 
$dQ=\e(f)d\P$ with some $f\in L^2([0,\infty))$. 
\begin{proposition}\label{cor:3}%
$\bigcap_n \Fn_\infty$ is trivial if and only if for all $Q\sim \P$
\begin{equation}\label{eq:EQ to 0}
\E[Q]{\hn_s\cond[\big]\Fn_s}\to0, \quad\text{$\P$--almost surely, for almost all $s>0$}.
\end{equation}
\end{proposition}
\begin{proof}In the proof we mostly work with simple equivalent
  measures, and obtain the conclusion of the  ``only if'' part by
  approximation. 

  First we get a formula for 
  $\E{\e(f)\cond[\big]\Fn_\infty}$ when $f\in L^2([0,\infty))$ and then we apply
  Proposition \ref{prop:1}.
  
  So for the simple equivalent measure
  $dQ=\e(f)d\P$, let the density process be denoted by
  $Z_t=\E{\e(f)\cond\F_t}$. Then $dZ_t=Z_tf(t)d\beta_t$ and by Lemma \ref{lem:2} 
  \begin{multline*}
    \Zn_\infty=\E{\e(f)\cond[\big]\Fn_\infty}=
    \E{1+\int_0^\infty Z_tf(t)d\beta_t\cond[\bigg]\Fn_\infty}\\=
    1+\int_0^\infty f(t)\E{Z_t\hn_t\cond[\big]\Fn_t}d\bn_t.
  \end{multline*}
  By the Bayes rule $\E{Z_t\hn_t\cond\Fn_t}=\E[Q]{\hn_t\cond\Fn_t}\Zn_t$.
  That is, with $\xin_t=\E[Q]{\hn_t\cond\Fn_t}$ and $\Mn=\int \xin_s
  f(s)d\bn_s$ we can write  
  \begin{displaymath}
    \E{\e(f)\cond\Fn_\infty}=\exp{\Mn_\infty-\frac12\q{\Mn}_\infty}.
  \end{displaymath}
  When \eqref{eq:EQ to 0} holds then $\q{\Mn}_\infty\to0$ in
  $L^1(\P)$, $\Mn_\infty\to0$ in $L^2(\P)$, hence $\ln \Zn_\infty\to
  0$ in $L^1(\P)$. Since $\Zn_\infty=\E{\e(f)\cond\Fn_\infty}$ converges
  almost surely we get  that its limit is 1. This is true for all
  $f\in L^2[0,\infty)$ and by Proposition \ref{prop:1} we
  obtain that the tail $\sigma$--field $\bigcap_n\Fn_\infty$ is trivial.

  For the converse we prove below that when $\bigcap_n\Fn_\infty$ is
  trivial then  for each  $f\in L^2 [0,\infty)$ 
  \begin{equation}\label{eq:E hn to 0}
   f(s)\E{\e(f)\hn_s\cond[\big]\Fn_s}\to0,\quad\text{almost surely, for almost
     all $s>0$}. 
  \end{equation}
  Then we  consider
  \begin{displaymath}
    \cH_s=\set*{\xi\in L^1(\P)}{\text{$\E{\xi\hn_s\cond[\big]\Fn_s}\to0$ in
        $L^1(\P)$}}, \quad s>0.
  \end{displaymath}
  $\cH_s$ is obviously a closed subspace of $L^1(\P)$.  It is possible
  to choose $D=\smallset{f_1,f_2,\dots}\subset L^2 ([0,\infty))$, a
  countable set of 
  deterministic, nowhere vanishing 
  functions, such that 
  the linear
  hull of $\set*{\e(f)}{f\in D 
  }$ is dense in $L^1(\P)$. 
  Finally let
  \begin{displaymath}
    \cT=\set*{s>0}{\forall f\in D,\,
      \E{\e(f)\hn_s\cond[\big]\Fn_s}\to0}.
  \end{displaymath}
  Then $\cT$ has full Lebesgue measure within $[0,\infty)$ and for
  $s\in\cT$ we obviously have $\cH_s=L^1(\P)$. For $s\in\cT$
  \eqref{eq:EQ to 0} follows, by considering $\xi=dQ/d\P$.  
 
  It remains to show that
  \begin{equation}\label{eq:F triv}
    \bigcap \Fn_\infty\quad\text{is trivial}
  \end{equation}
  implies \eqref{eq:E hn to 0}. So we fix $f$ and use the 
  notation $Q$, $\xin$, $\Mn$ introduced at the beginning of the proof. 
  Note that 
  $(\abs{\xin_s},\Fn_s)_{n\geq0}$ is a reversed  $Q$-submartingale
  for each fixed $s$. Hence $\abs{\xin_s}$ is
  convergent almost surely (both under $\P$ and $Q$ by their
  equivalence) and the limit is $\bigcap_{n}\Fn_s\subset
  \bigcap_{n}\Fn_\infty$ measurable. Since $\bigcap_{n}\Fn_\infty$ is
  trivial there is a deterministic function $g$ such that
  $\abs{\xin_s}\to g(s)$ almost surely for almost all $s$. Obviously 
  $0\leq g(s)\leq 1$.

  Another implication of \eqref{eq:F triv} is that 
  \begin{equation}\label{eq:ln Z to 0}
    \ln\E{\e(f)\cond[\big]\Fn_\infty}=\Mn_\infty-\frac12\q{\Mn}_\infty \to
    0,\quad\text{almost surely}. 
  \end{equation}
  Here 
  \begin{displaymath}
    \q{\Mn}_\infty\to\sigma^2=\int_0^\infty \zjel{f(s)g(s)}^2ds,
    \quad\text{almost surely}
  \end{displaymath}
  and we will see that $\Mn_\infty$ has normal limit with expectation
  zero and variance $\sigma^2$. Then \eqref{eq:ln Z to 0} can only
  hold if $\sigma^2=0$ which obviously implies \eqref{eq:E hn to 0}.

  To finish the proof we write $\Mn_\infty$ as
  \begin{displaymath}
    \Mn_\infty=\int_0^\infty f(s)g(s)\sign(\xin_s)d\bn_s+
    \int_0^\infty f(s)\zjel{\xin_s-g(s)\sign(\xin_s)}d\bn_s.
  \end{displaymath}
  Here the law of the first term is normal $N(0,\sigma^2)$ not depending on $n$,
  while the second term goes to zero in $L^2(\P)$.   
\end{proof}

To finish the proof of Proposition \ref{prop:2} assume that 
\eqref{eq:cond} holds, that is 
  \begin{equation*}
    \lim_{n\to\infty}\E{\hn_t\cond[\big]\Fn[0]_{rt}\vee\Fn_t}=0,
    \quad\text{for almost
      all $t>0$ and $r\in[0,1)$}.
  \end{equation*}
Fix a $Q\sim\P$ and denote by $Z_t=\frac{dQ|_{\Fn[0]_t}}{d\P|_{\Fn[0]_t}}$ the 
density process. By the Bayes formula it is enough to show that 
\begin{displaymath}
  \E{Z_t\hn_t\cond[\big]\Fn_t}\to 0.
\end{displaymath}

Since $\abs{\hn}\leq1$ we have the next estimate
\begin{displaymath}
  \norm*{\E{Z_t\hn_t\cond[\big]\Fn[0]_{rt}\vee\Fn_t}-
    \E{Z_{rt}\hn_t\cond[\big]\Fn[0]_{rt}\vee\Fn_t}}_{L^1}\leq 
  \norm*{Z_{t}-Z_{rt}}_{L^1}, 
\end{displaymath}
and by \eqref{eq:cond}
\begin{displaymath}
  \E{Z_{rt}\hn_t\cond[\big]\Fn[0]_{rt}\vee\Fn_t}=
  Z_{rt}\E{\hn_t\cond[\big]\Fn[0]_{rt}\vee\Fn_t}\to0
\end{displaymath}
almost surely and in $L^1$. That is,
\begin{displaymath}
  \limsup_{n\to\infty}\norm*{\E{Z_t\hn_t\cond[\big]\Fn_t}}_{L^1}\leq 
  \inf_{r\in[0,1)}\norm*{Z_{t}-Z_{rt}}_{L^1}=0.
\end{displaymath}
This means that the 
limit of the reversed submartingale
$\abs[\big]{\E[Q]{\hn_t\cond[\big]\Fn_t}}$ is zero and $T$ is exact by
Proposition  
\ref{cor:3}. This completes the proof of Proposition \ref{prop:2}.

\acknowledgement{The author thanks 
Michel Emery for reading the first version of
this note and  offering helpful comments,  and the referee
for suggesting a simplification in the proof of Proposition \ref{cor:3}.
}

\def\doi#1{doi:~\href{http://dx.doi.org/#1}{\nolinkurl{#1}}}

\end{document}